\documentclass[11pt]{amsart}

\usepackage{amsmath}
\usepackage{fullpage}
\usepackage{xspace}
\usepackage[psamsfonts]{amssymb}
\usepackage[latin1]{inputenc}
\usepackage{graphicx,color}
\usepackage[curve]{xypic} 
\usepackage{hyperref}
\usepackage{graphicx}

\usepackage{amsmath}%
\usepackage{amsthm}%
\usepackage{amscd}
\usepackage{amsfonts}%
\usepackage{amssymb}%
\usepackage{graphicx}

\usepackage{mathrsfs}

\usepackage{tikz}
\usetikzlibrary{matrix,arrows}

\newtheorem{theorem}{Theorem}[section]

\newtheorem{conjecture}[theorem]{Conjecture}
\newtheorem{corollary}[theorem]{Corollary}

\newtheorem{lemma}[theorem]{Lemma}

\newtheorem{problem}{Problem}[section]
\newtheorem{proposition}[theorem]{Proposition}

\theoremstyle{remark}

\numberwithin{equation}{section}

\newcommand{\Ccal}{\mathscr{C}}

\newcommand{\Ecal}{\mathscr{E}}

\newcommand{\Ncal}{\mathscr{N}}
\newcommand{\Ocal}{\mathscr{O}}
\newcommand{\Pcal}{\mathscr{P}}

\newcommand{\Vcal}{\mathscr{V}}

\newcommand{\Pro}{\mathbb{P}}

\newcommand{\Z}{\mathbb{Z}}
\newcommand{\C}{\mathbb{C}}

\newcommand{\Q}{\mathbb{Q}}
\newcommand{\R}{\mathbb{R}}
\newcommand{\N}{\mathbb{N}}

\newcommand{\supp}{\mathrm{supp}}

\newcommand{\Pole}{\mathrm{Pole}}
\newcommand{\Zero}{\mathrm{Zero}}
\newcommand{\Deg}{\mathrm{Deg}}
\newcommand{\mult}{\mathrm{mult}}
\newcommand{\Div}{\mathrm{Div}}
\newcommand{\Der}{\mathrm{Der}}

\newcommand{\ddim}{\mathrm{ddim}}

\newcommand{\alg}{\mathrm{alg}}

  \DeclareFontFamily{U}{wncy}{}
    \DeclareFontShape{U}{wncy}{m}{n}{<->wncyr10}{}
    \DeclareSymbolFont{mcy}{U}{wncy}{m}{n}
    \DeclareMathSymbol{\Sha}{\mathord}{mcy}{"58}

\begin{document}
\title[Non-Diophantine sets in rings of functions]{Non-Diophantine sets in rings of functions}

\author{Natalia Garcia-Fritz}
\address{ Departamento de Matem\'aticas,
Pontificia Universidad Cat\'olica de Chile.
Facultad de Matem\'aticas,
4860 Av.\ Vicu\~na Mackenna,
Macul, RM, Chile}
\email[N. Garcia-Fritz]{natalia.garcia@mat.uc.cl}%

\author{Hector Pasten}
\address{ Departamento de Matem\'aticas,
Pontificia Universidad Cat\'olica de Chile.
Facultad de Matem\'aticas,
4860 Av.\ Vicu\~na Mackenna,
Macul, RM, Chile}
\email[H. Pasten]{hector.pasten@mat.uc.cl}%

\author{Thanases Pheidas}
\address{Department of Mathematics and Applied Mathematics and Laboratory of Theoretical Mathematics, University of Crete, Greece}
\email[T. Pheidas]{pheidas@uoc.gr}%

\thanks{N.G.-F. was supported by ANID Fondecyt Regular grant 1211004 from Chile. H.P was supported by ANID (ex CONICYT) FONDECYT Regular grant 1190442 from Chile. The three authors were supported by the NSF Grant No. DMS-1928930 while at the MSRI, Berkeley, in Summer 2022.}
\date{\today}

\begin{abstract} Except for a limited number of cases, a complete classification of the Diophantine sets of polynomial rings and fields of rational functions seems out of reach at present. We contribute to this problem by proving that several natural sets and relations over these structures are not Diophantine.
\end{abstract}

\maketitle



\section{Introduction} 

Let $R$ be a (commutative, unitary) ring and let $n\ge 1$. A set $X\subseteq R^n$ is \emph{Diophantine} over $R$ if there are polynomials $F_1,...,F_r\in R[x_1,...,x_n, y_1,...,y_m]$ such that
$$
X = \{{\bf a}\in R^n : \mbox{ there is }{\bf b}\in R^m\mbox{ such that }F_j({\bf a},{\bf b})=0\mbox{ for each }1\le j\le r\}.
$$
Equivalently, $X\subseteq R^n$ is Diophantine over $R$ if it is positive existentially definable over $R$ with parameters from $R$.

After the work of Davis, Putnam, Robinson \cite{DPR} and Matiyasevich \cite{Mat} we have a satisfactory description of the Diophantine subsets of $\Z^n$ over $\Z$; they are precisely the listable subsets. On the other hand, we know very little about the Diophantine subsets of $\Q^n$ over $\Q$; in particular, it is not known whether $\Z$ is Diophantine in $\Q$.

There is a well-known analogy between numbers and functions, where polynomial rings are analogous to $\Z$ and fields of rational functions are analogous to $\Q$. Motivated by this analogy, in this work we investigate the Diophantine sets of polynomial rings and fields of rational functions, by producing several natural subsets that turn out to be non-Diophantine.

\emph{Notation:} In this article, $t$ and $z$ denote independent variables that are transcendental over the relevant rings and fields (we use them to define  polynomials, rational functions, and power series).


\subsection{Polynomials}\label{SecPolyIntro} In the setting of polynomials let us recall the results of Demeyer \cite{DemeyerFp, DemeyerNF} stating that if $k$ is a finite field or a number field, then the polynomial ring $k[z]$ has the DPRM property: its listable sets are the same as its Diophantine sets. (See \cite{Pasten} for a more detailed discussion on other structures having the DPRM property.) Thus, in these cases we have a complete description of the Diophantine sets over such polynomial rings. However, in general, not much is known about the Diophantine sets of polynomial rings over other fields. For instance we have the following open problem pointed out to us by Scanlon:
 
\begin{problem}\label{ProblemDeg} Determine whether the set
$$
\{(f,g)\in \C[z]^2 : \deg(f)=\deg(g)\}
$$
is Diophantine over $\C[z]$ or not.
\end{problem}

Here we remark that the analogous set with $\C$ replaced by $\R$ is in fact Diophantine over $\R[z]$ by results of Denef, see Section \ref{SecDioph}.

Let $A$ be a ring and let $R=A[z]$ be the corresponding polynomial ring in one variable. For $X\subseteq R$ and an integer $\alpha\ge 0$ we define $X_\alpha = \{f\in X : \deg(f)\le \alpha\}$. Note that by taking coefficients we have an identification $R_\alpha=A^{\alpha+1}$ and in this way we can talk about Diophantine subsets of $R_\alpha$ over $A$. Our first result is the following general observation that will be useful for detecting some simple non-Diophantine sets over polynomial rings.
\begin{theorem}\label{ThmPoly} Let $A$ be a ring.  Let $D\subseteq A[z]$ be Diophantine over $A[z]$ and let $\alpha \ge 0$. Then $D_\alpha\subseteq A^{\alpha+1}$  is a countable union of Diophantine sets over $A$.
\end{theorem}
For the proof of this result and its applications to be discussed below, see Section \ref{SecPoly}.

 Given a set $S\subseteq A[z]$ we define 
$$
\Zero(S)=\{a\in A : f(a)=0\mbox{ for some }f\in S\}\subseteq A.
$$
A first consequence of Theorem \ref{ThmPoly} is the following.
\begin{corollary}\label{CoroZeros} Let $A$ be a ring. Let $D\subseteq A[z]$ be a Diophantine set and let $\alpha\ge 0$. Then $\Zero(D_\alpha)$ and $\Zero(D)$ are the countable union of Diophantine sets over $A$.
\end{corollary}
Let us mention some other consequences of Theorem \ref{ThmPoly}.
\begin{corollary}\label{CoroDeleteZQ} The set $\C-\Z\subseteq \C= \C[z]_0$ is not Diophantine over $\C[z]$, and the set $\R-\Q\subseteq \R= \R[z]_0$ is not Diophantine over $\R[z]$. 
\end{corollary}
This corollary should be contrasted with the fact that $\Z$ is Diophantine in both $\C[z]$ and $\R[z]$, and that $\R-\Z$ is Diophantine in $\R[z]$. See Section \ref{SecDioph}.

Let us recall that a field $L$ is \emph{large} if the following holds for every curve $C$ over $L$: If $C$ has a smooth $L$-rational point, then it has infinitely many. See \cite{PopLarge}. Here is another application of Theorem \ref{ThmPoly}.

\begin{corollary}\label{CoroSubfield}
Let $L$ be a large uncountable  field of characteristic $0$ and let $K\subseteq L$ be a properly contained uncountable sub-field. Then $K$ is not Diophantine in $L[z]$.
\end{corollary}
This corollary strengthens  some instances of a result by Fehm (cf. Corollary 9 in \cite{Fehm}), namely, that if $L$ is a large field of characteristic $0$ and $K\ne L$ is a sub-field, then $K$ is not Diophantine over $L$. Note that in Corollary \ref{CoroSubfield} the hypothesis that $K$ be uncountable cannot be removed; by work of Denef we know that $\Q$ is Diophantine in $k[z]$ for every field $k$ of characteristic $0$, see Section \ref{SecDioph}.

As special cases of Corollary \ref{CoroSubfield}, we have
\begin{corollary}
Let $k$ be a field of characteristic $0$, let $L=k((t))$, and let $n\ge 2$. Consider the sub-field $K_n=k((t^n))\subseteq L$. Then $K_n$ is not Diophantine over $L[z]$.
\end{corollary}
Note that if $k$ is an algebraically closed field of characteristic $p$, then by taking $p$-th powers we see that $k((t^p))$ is Diophantine in $k((t))$, hence, in $k((t))[z]$. 
\begin{corollary}
Let $k$ be an uncountable field of characteristic $0$ and let $L=k((t))$. The sub-field $k\subseteq L$ is not Diophantine over $L[z]$.
\end{corollary}
The condition that $k$ be uncountable cannot be removed: $\Q$ is Diophantine in $L[z]$ for every field $L$ of characteristic $0$ (see Section \ref{SecDioph}), in particular for $L=\Q((t))$. Also, it is worth pointing out that if $k$ is a large field then $k$ is Diophantine in $k(t)[z]$, see Section \ref{SecDioph}.

Finally, let us discuss an application of Theorem \ref{ThmPoly} of a different kind. For a subset $S\subseteq \R[z]$ we define its real closure $\overline{S}\subseteq \R[z]$ as follows: For each $\alpha$ we let $\overline{S_\alpha}\subseteq \R[z]_\alpha= \R^{\alpha+1}$ be the real closure of $S_\alpha$ in $\R^{\alpha+1}$, and let $\overline{S}=\cup_\alpha \overline{S_\alpha}$. Given $D\subseteq \R[z]$ Diophantine one can ask whether its real closure is also Diophantine. We show that in general this is not the case.

\begin{corollary}\label{CoroClosure} There is a set $D\subseteq \R[z]_1$ which is Diophantine over $\R[z]$ but its real closure $\overline{D}\subseteq \R[z]_1$ is not.
\end{corollary}


\subsection{Rational functions} Let $L$ be an uncountable large field of characteristic $0$.

Regarding rational functions, we recall that Koll\'ar \cite{Kollar} proved that $L[z]$ is not Diophantine in $L(z)$, and that for any $g\in L(z)$ with degree at least $2$ (as a map $g:\Pro^1\to \Pro^1$) the sub-field $L(g)\subseteq L(z)$ is not Diophantine in $L(z)$. Koll\'ar actually proved a general result about Diophantine subsets of $L(z)$ from which these two applications are derived. Using the results of Koll\'ar we prove that other natural sets are not Diophantine over $L(z)$.

For a field $k$ and $g\in k(z)$ let $\deg(g)$ be the degree of the map $g:\Pro^1\to \Pro^1$ induced by $g$; thus, if $g\in k[z]$ and $g\ne 0$ then $\deg(g)$ as defined here agrees with the usual definition of the degree of a polynomial (note that with this definition $\deg(0)=0$ as for any other constant function). 

In connection with Problem \ref{ProblemDeg} we prove
\begin{theorem}\label{ThmDeg} The set
$$
\Deg=\{(f,g)\in L(z)^2 : \deg(f)=\deg(g)\}
$$
is not Diophantine over $L(z)$.
\end{theorem}
This should be compared with the following fact. There is another way to extend the degree from $L[z]$ to $L(z)$, namely, by defining $\deg^*(g)=-v_\infty(g)$ where $v_\infty$ is the valuation at the point at infinity of $\Pro^1$. Then, it follows from results of Denef that 
$$
\{(f,g)\in \R(z)^2 : \deg^*(f)=\deg^*(g)\}
$$
is Diophantine over $\R(z)$; the analogous result over $\C$ is not known. See Section \ref{SecDioph}.

The degree $\deg(g)$ equals the number of poles of $g$ in $\Pro^1$ over $L^{\alg}$, counting multiplicity. One can ask about the analogous concept without counting multiplicity, a notion that naturally appears in function field arithmetic in the form of the truncated counting function \cite{Vojta}. For $n\ge 1$ let us define 
$$
\Pcal_n=\{f\in L(z) : f \mbox{ has at most }n\mbox{ poles in }\Pro^1\mbox{ over }L^{\alg}, \mbox{ without counting multiplicity}\}.
$$
\begin{theorem}\label{ThmN1} For each $n\ge 1$ the set $\Pcal_n$ is not Diophantine over $L(z)$.
\end{theorem}
Consider the valuation ring at $\infty\in \Pro^1$
$$
\Ocal_{\infty}(L) = \{f\in L(z) : v_\infty(f)\ge 0\}.
$$
We will be concerned with the question of whether $\Ocal_{\infty}(L)$ is Diophantine in $L(z)$ or not. For technical reasons it is more convenient to work instead with the set 
$$
\Vcal(L)=\{f\in L(z) : v_\infty(f)\le 0\}
$$
but, of course, $\Ocal_{\infty}(L)$ is Diophantine in $L(z)$ if and only if $\Vcal(L)$ is (by taking reciprocals).

 By results of Denef (see Section \ref{SecDioph}) the set $\Ocal_{\infty}(\R)$ is Diophantine in $\R(z)$, but the following remains as  an important open problem:
\begin{problem}\label{ProblemVal} Is $\Ocal_{\infty}(\C)$ Diophantine in $\C(z)$? Equivalently, is $\Vcal(\C)$ Diophantine in $\C(z)$?
\end{problem}
By work of Denef \cite{Denef}, a positive answer would yield a negative solution to the analogue of Hilbert's tenth problem over $\C(z)$, which is one of the main open problems in the area (cf. p.56 and Section 2 in \cite{Survey}). We prove a result approaching a negative answer to Problem \ref{ProblemVal}.

 For $0\le \epsilon \le 1$ let us define $\Vcal_\epsilon(L)\subseteq L(z)^\times$ as the set of all rational functions $f\in L(z)^\times$ that can be written as $f=p/q$ with $p,q\in L[z]$ non-zero coprime polynomials satisfying
$$
\deg(q)\le (1-\epsilon) \deg(p).
$$
Note that $\Vcal_1(L)=L[z]-\{0\}$ and $\Vcal_0(L)=\Vcal(L)$. Furthermore, for each $0\le \eta \le  \epsilon \le 1$ we have $\Vcal_\eta(L)\supseteq \Vcal_\epsilon(L)$, and the (nested) union of the sets $\Vcal_\epsilon(L)$ for $0< \epsilon\le 1$ is
$$
\bigcup_{0< \epsilon \le 1} \Vcal_\epsilon(L) = \{f\in L(z) : v_\infty(f)< 0\} = z\cdot \Vcal(L).
$$
Thus, $\Vcal_0(L)=\Vcal(L)$ and in the limit $\epsilon\to 0$ the sets $\Vcal_\epsilon(L)$ converge to $z\cdot \Vcal (L)$. We prove:
\begin{theorem}\label{ThmVal} For each $0< \epsilon\le 1$ the set $\Vcal_\epsilon(L)$ is not Diophantine in $L(z)$.
\end{theorem}
So, although we cannot treat the limit case $\epsilon =0$ which corresponds to the valuation ring $\Ocal_\infty(L)$, we can approximate it as much as we want by letting $\epsilon\to 0$, obtaining that the corresponding sets $\Vcal_\epsilon(L)$ are not Diophantine.

As it will be clear from the proofs, this last discussion on the valuation at the point $\infty\in \Pro^1$ can be repeated at any other point of $\Pro^1$ with the appropriate changes in the notation, but we decided to focus just in this case to simplify the discussion.


\subsection{A conjectural picture for complex rational functions} In this section we focus on complex rational functions. Let us recall the following conjecture proposed by Koll\'ar \cite{Kollar}.

\begin{conjecture}[Koll\'ar]\label{ConjKollar}
Let $D\subseteq \C(z)$ be a Diophantine set such that for infinitely many integers $n\ge 0$ we have that $D$ contains a Zariski open subset of $\C[z]_n=\C^{n+1}$. Then $\C(z)-D$ is finite.
\end{conjecture}

So far, this conjecture has remained unexplored. Here we would like to record several simple consequences that suggest that Koll\'ar's conjecture, if correct, imposes strong restrictions on the Diophantine sets of $\C(z)$ and supports some non-analogies with $\Q$. See Section \ref{SecKollarConjProofs}.

\begin{proposition}\label{PropConj}
If Conjecture \ref{ConjKollar} is true, then the following subsets of $\C(z)$ are not Diophantine:
\begin{itemize}
\item[(i)] Non-constants: $T=\C(z)-\C$;
\item[(ii)] Derivatives: $\Der = \{f' : f\in \C(z)\}$ where $f'$ is the derivative of $f$;
\item[(iii)] Valuations: $\Ocal_\infty(\C) = \{f\in \C(z) : v_\infty(f)\ge 0\}$ and more generally, for any $x\in \Pro^1$ the set $\Ocal_x(\C)= \{f\in \C(z) : v_x(f)\ge 0\}$ where $v_x$ is the valuation at $x$;
\item[(iv)] Non-squares: $\Ncal=\C(z)-\{f^2 : f\in \C(z)\}$;
\item[(v)] Non-polynomials: $\C(z)-\C[z]$.
\end{itemize}
\end{proposition}

We remark that the analogues for $\Q$ of the sets defined in items (iii), (iv) and (v) are in fact Diophantine, see \cite{Poonen, KoeZQ}.

For $n\in \{1,2,3,...,\infty\}$ and a finite set of points $S\subseteq \Pro^1(\C)$ we define the set of \emph{Campana points} in $\C(z)$ as
$$
\Ccal_{S,\ell} = \{f\in \C(z): \forall \alpha\in \Pro^1-S,\mbox{ if } v_\alpha(f)\le -1\mbox{ then } v_\alpha(f)\le -\ell\}.
$$ 
These sets interpolate between $\C[z]$ and $\C(z)$; in fact, taking $\ell=\infty$ and $S=\{\infty\}\subseteq \Pro^1$ we have $\Ccal_{S,\ell}=\C[z]$, while taking $\ell=1$ and $S$ arbitrary we get $\Ccal_{S,\ell}=\C(z)$.

 Campana points in the number field and function field setting have been the subject of considerable attention in recent years, in connection with Campana's conjectures and extensions of Manin's conjecture, see for instance \cite{Abr, AV, BVV, BY, Cam, GF, RTW, PSTV, VV} for recent developments and for the precise definitions (and generalizations) in the number field and function field setting. Recently,  De Rasis has shown that Campana points over $\Q$ admit a $\forall\exists$ first order definition (upcoming work), and one can ask whether they admit a Diophantine (i.e. existential) definition. In the function field setting we have:
\begin{proposition}\label{PropCampana}
Let $\ell\in \{1, 2, 3,...,\infty\}$ and let $S\subseteq \Pro^1$ be a finite set. If Conjecture \ref{ConjKollar} is true, then $\Ccal_{S,\ell}$ is not Diophantine, unless we are in one of the following two cases:
\begin{itemize}
\item[(a)] $\ell=1$, in which case $\Ccal_{S,\ell}=\C(z)$, or
\item[(b)] $\ell=\infty$ and $S=\emptyset$, in which case $\Ccal_{\emptyset, \infty}=\C$.
\end{itemize}
\end{proposition}

We conclude this introduction by pointing out that the analogue of Koll\'ar's conjecture in positive characteristic is false, although for trivial reasons (see Section \ref{SecKollarConjProofs}).

\begin{proposition}\label{PropPosChar} Let $p$ be a prime and let $k$ be an algebraically closed field of characteristic $p$. There is a Diophantine subset $D\subseteq k(z)$ such that $k(z)-D$ is infinite and for every $n\ge 1$ the set $D$ contains a Zariski open subset of $k[z]_n=k^{n+1}$.
\end{proposition}


\section{Some Diophantine sets}\label{SecDioph}

In this section we collect several known examples of Diophantine sets in polynomial rings and fields of rational functions.
\begin{lemma}\label{LemmaCtPoly} Let $k$ be a field. For each integer $\alpha\ge 0$  the set $k[z]_{\alpha}\subseteq k[z]$ is Diophantine over $k[z]$.
\end{lemma}
\begin{proof} One sees that the field $k$ is Diophantine by taking $0$ or invertible elements. Then $k[z]_\alpha$ is obtained as $k$-linear combinations of $1,z,z^2,...,z^\alpha$.
\end{proof}

The next result is a special case of the general theorems in \cite{MoretBailly}.

\begin{lemma} Let $A$ be a noetherian integral domain. The set $A[z]-\{0\}$ is Diophantine in $A[z]$.
\end{lemma}

In general, it is a difficult question whether a field $k$ is Diophantine in $k(z)$; for instance, the case $k=\Q$ remains open. Nevertheless one has (cf. \cite{KoeLarge})

\begin{lemma}\label{LemmaCtRat} Let $L$ be a large field. Then $L$ is Diophantine in $L(z)$.
\end{lemma}

\begin{corollary}\label{LemmaCtRatPoly} Let $L$ be a large field. Then $L$ is Diophantine in $L(t)[z]$.
\end{corollary}
\begin{proof} By Lemma \ref{LemmaCtPoly} with $k=L(t)$ and Lemma \ref{LemmaCtRat}.
\end{proof}

The next result directly follows from work of Denef \cite{Denef}.

\begin{lemma} Let $k$ be a field of characteristic $0$. Then $\Z$  and $\Q$ are Diophantine in $k[z]$.
\end{lemma}

\begin{corollary} The set $\R-\Z$ is Diophantine in $\R[z]$.
\end{corollary}
\begin{proof} We have that $\R$ and $\Z$ are both Diophantine in $\R[z]$, and so is the relation $<$ in $\R$ as a subset of $\R[z]$. Then we observe that for $a\in \R$ we have that $a\notin \Z$ if and only if $\exists n\in \Z, n<a<n+1$.
\end{proof}

The next result is also due to Denef, which can be obtained after a simple modification of the proof of Lemma 3.5 in \cite{Denef}.

\begin{lemma} The valuation ring $\Ocal_\infty(\R)=\{f\in \R(z) : v_{\infty}(f)\ge 0\}\subseteq \R(z)$ is Diophantine in $\R(z)$.
\end{lemma}

\begin{corollary} The set 
$$
\{(f,g) \in \R(z)^2 : v_\infty(f)=v_\infty(g)\}
$$
is Diophantine in $\R(z)$.
\end{corollary}

\begin{corollary} The set 
$$
\{(f,g) \in \R[z]^2 : \deg(f)=\deg(g)\}
$$
is Diophantine in $\R[z]$.
\end{corollary}

By the previous corollary and results of Demeyer \cite{DemeyerNF} we deduce

\begin{lemma}\label{LemmaQzinRz} The ring $\Q[z]$ is Diophantine in $\R[z]$.
\end{lemma}

Regarding $\Q[z]$, the following is proved in \cite{DemeyerNF}, building on work of Denef \cite{DenefZT}.

\begin{lemma}\label{LemmaDPRMinQz} The ring $\Q[z]$ has the DPRM property: listable sets in $\Q[z]$ are the same as Diophantine sets over $\Q[z]$. 
\end{lemma}


\section{Non-Diophantine sets of polynomials}\label{SecPoly}

In this section we prove the results from Section \ref{SecPolyIntro}.

\begin{proof}[Proof of Theorem \ref{ThmPoly}] Let $R=A[z]$. There are polynomials $F_1,..., F_r\in R[x,y_1,...,y_m]$ such that 
$$
D = \{a\in R : \mbox{ there is }{\bf b}\in R^m\mbox{ such that }F_j(a,{\bf b})=0\mbox{ for each }1\le j\le r\}.
$$
For any integers $\alpha,\beta \ge 0$ let $G_{\alpha,\beta}$ be the set of all solutions to the system of equations 
$$
\begin{cases}
F_1(x,{\bf y})=0\\
\vdots\\
F_r(x,{\bf y})=0
\end{cases} 
$$
in $A[z]_{\alpha}\times A[z]_{\beta}^m$. By considering coefficients let us make the identification of $A$-modules $A[z]_{\alpha}\times A[z]_{\beta}^m=A^{N(\alpha,\beta)}$ where $N(\alpha,\beta)=\alpha+1 +  (\beta+1)m$. Thus, $G_{\alpha,\beta}\subseteq A^{N(\alpha,\beta)}$ and by writing  out the defining condition of $Y_{\alpha, \beta}$ in terms of coefficients of polynomials, we see that $G_{\alpha,\beta}$ is Diophantine over $A$ (in fact, it is quantifier-free definable).

Let $\pi_{\alpha,\beta}: A^{N(\alpha,\beta)}\to A[z]_\alpha=A^{\alpha+1}$ be the $A$-linear projection onto the first $\alpha+1$ coordinates. Then
$$
D_\alpha = \bigcup_\beta \pi_{\alpha,\beta}(G_{\alpha,\beta})\subseteq A^{\alpha+1}=A[z]_\alpha
$$
and the result follows.
\end{proof}

\begin{proof}[Proof of Corollary \ref{CoroZeros}] It suffices to show the case of $D_\alpha$. By Theorem \ref{ThmPoly} we can write $D_\alpha=\cup_{\beta=0}^\infty E_{\alpha,\beta}$ where $\beta$ varies over integers and each $E_{\alpha,\beta}\subseteq A^{\alpha+1}$ is Diophantine over $A$. Let us define
$$
X_{\alpha,\beta} = \{a\in A : \exists f\in E_{\alpha,\beta}, f(a)=0\}\subseteq A.
$$
Expressing $f(a)$ in terms of the coefficients of $f$ and powers of $a$ we see that $X_{\alpha,\beta}\subseteq A$ is Diophantine over $A$. The result follows from the fact that $\Zero(D_\alpha)=\cup_\beta X_{\alpha,\beta}$.
\end{proof}

\begin{proof}[Proof of Corollary \ref{CoroDeleteZQ}] Here we invoke Tarski's classical result on elimination of quantifiers over $\C$ and $\R$. A countable union in $\C$ of Diophantine sets over $\C$ is either countable or cofinite. A countable union in $\R$ of Diophantine sets over $\R$ is a countable union of points and intervals. Hence the result.
\end{proof}

\begin{proof}[Proof of Corollary \ref{CoroSubfield}] 
Suppose that $K$ is Diophantine in $L[z]$. Then $K=\cup_{\beta\in \N} X_\beta$ where $X_\beta\subseteq L$ is Diophantine over $L$. Let $L^{(\beta)}$ be the subfield of $L$ generated by $X_\beta$, thus, $K=\cup_{\beta\in \N} L^{(\beta)}$. By cardinality considerations, there is $\gamma$ such that $X_\gamma$ is infinite. Thus, $L^{(\gamma)}$ is a subfield of $L$ generated by the infinite Diophantine subset $X_\gamma\subseteq L$. From Theorem 2 in \cite{Fehm} (see also \cite{Anscombe}) it follows that $L^{(\gamma)}=L$, hence $K=L$; contradiction.
\end{proof}

We remark that a version of Corollary \ref{CoroSubfield} in positive characteristic (taking into account $p$-th power sub-fields) can be proved using results from \cite{Anscombe} instead of \cite{Fehm}.

The following lemma is well-known but we were unable to find a proof in the literature, so we provide one.
\begin{lemma}\label{LemmaAnalytic} There is a transcendental analytic function $f:\R\to \R$ with the properties that $f(\Q)\subseteq \Q$ and that the map $f|_\Q:\Q\to \Q$ is computable.
\end{lemma}
\begin{proof} Let $q_1,q_2,...$ be a computable enumeration of the \emph{squares} in $\Q$. Let 
$$
P_n(z)=(q_1-z^2)\cdots (q_n-z^2)\in \Q[z]
$$ 
and let $A_n$ be a computable sequence of positive integers with $|P_n(z_0)|< A_n(|z_0|^{2n}+1)$ for all $z_0\in \C$ and all $n\ge 1$; a suitable choice of such integers $A_n$ can be easily computed from the sequence $q_n$. Then consider
$$
f(z)=\sum_{n=1}^\infty \frac{1}{(2n)!A_n}P_n(z).
$$
By comparison with the exponential function we see that $f$ is an analytic map $\R\to \R$; in fact, it defines an analytic function even over $\C$. Also, it maps $\Q$ to $\Q$ and $f|_\Q:\Q\to \Q$ is computable. It only remains to check that it is transcendental. For this, let $g:\C\to \C$ be the function
$$
g(t)=f(it)=\sum_{n=1}^\infty \frac{1}{(2n)!A_n}(q_1+t^2)\cdots (q_n+t^2)
$$
where $i=\sqrt{-1}\in \C$. The power series expansion of $g(t)$ is infinite with positive coefficients in even degree. Since $f(z)$ is analytic on $\C$, if it were algebraic it would be a polynomial and so would be $g(t)$, which is not the case.
\end{proof}

\begin{proof}[Proof of Corollary \ref{CoroClosure}] Let $f$ be as in Lemma \ref{LemmaAnalytic}. Since $f|_\Q:\Q\to\Q$ is computable, by Lemmas \ref{LemmaQzinRz} and \ref{LemmaDPRMinQz} we see that the following set is Diophantine over $\R[z]$:
$$
D=\{(q,f(q)) : q\in \Q\}\subseteq \R^2=\R[z]_1.
$$
The real closure of $D$ is the graph of $f$, which we denote by $\Gamma_f$. By Theorem \ref{ThmPoly}, if $\Gamma_f\subseteq \R^2$ were Diophantine over $\R[z]$ then it would be a countable union of semi-algebraic sets over $\R$, which is not possible because $f$ is transcendental and analytic.
\end{proof}


\section{Non-Diophantine sets of rational functions} 

In this section we let $L$ be a large uncountable field of characteristic $0$.

If $E\ge 0$ is a divisor in $\Pro^1$ defined over $L^{\alg}$ of degree $\ell$, then it can be seen as an $L^{\alg}$-rational point in the symmetric power $S^\ell \Pro^1$ and every $L^{\alg}$-rational point in $S^\ell \Pro^1$ can be seen in this way. In particular, given an $r\ge 1$ we have the subvariety $r\cdot S^\ell\Pro^1\subseteq S^{r\ell}\Pro^1$ obtained by multiplying divisors by $r$.

For $f\in L(z)$ let $\Pole(f)$ be the divisor of poles of $f$ in $\Pro^1$ over $L^{\alg}$. Writing $\ell=\deg(f)$ we note that $\deg\Pole(f)=\ell$ and therefore $\Pole(f)\in S^\ell\Pro^1(L^{\alg})$. 

For a set $D\subseteq L(z)$ and an integer $\ell\ge 1$ we write 
$$
\Pole_\ell(D)=\{\Pole(f) : \deg(f)=\ell\}
$$
which is a subset of $S^\ell\Pro^1(L^{\alg})$.

If $V$ is a variety over $L^{\alg}$ and $X\subseteq V$ is a subset, the Zariski closure of $X$ is denoted by $\overline{X}\subseteq V$. 

For a set $D\subseteq L(z)$ we let $D_\alpha =\{f\in D : \deg(f)\le \alpha\}$. Note that $L(z)_\alpha$ consists of the fractions of coprime polynomials of degree up to $\alpha$, hence, considering coefficients we see that $L(z)_\alpha$ has the structure of an algebraic variety over $L$. 

The Diophantine dimension $\ddim(D)$ of a set $D\subseteq L(z)$ is the smallest  $d\in\{-1,0,1,2,...,\infty\}$ such that for each $\alpha\ge 0$, the set $D_\alpha$ is contained in a countable union of sub-varieties of $L(z)_\alpha$ defined over $L$ of dimension less than or equal to $d$.

Let us recall the main result of Koll\'ar's work \cite{Kollar}. We state it in a less general form which is nonetheless easier to use over the field of rational functions $L(z)$.

\begin{theorem}[Koll\'ar]\label{ThmKollar} Let $D\subseteq L(z)$ be a Diophantine subset over $L(z)$. If $\ddim(D)=\infty$, then there are $a\ge 0$, $P_a\in S^a\Pro^1(L^{\alg})$, and $r\ge 1$ such that for all $m\ge 1$ we have
$$
P_a+ r\cdot S^m\Pro^1 \subseteq \overline{\Pole_{a+rm}(D)}
$$
inside $S^{a+rm}\Pro^1$.
\end{theorem}

Let us record here a simple consequence.

\begin{corollary}\label{CoroAP}
Let $D\subseteq L(z)$ be a Diophantine subset over $L(z)$. If $\ddim(D)=\infty$, then the set
$$
\{\deg(f) : f\in D\}
$$
contains an infinite arithmetic progression.
\end{corollary}
\begin{proof} This is because $\deg(f)=\deg\Pole(f)$.
\end{proof}

The next result on elliptic curves is needed for the proof of Theorem \ref{ThmDeg}.

\begin{lemma} Consider the elliptic curve $\Ecal$ over $\Q(z)$ defined by the Weierstrass equation
\begin{equation}\label{EqnEC}
y^2 = x^3+zx+1
\end{equation}
and define the $\Q(z)$-rational point $P_1=(x,y)=(0,1)$. For $n\in \Z$ let $P_n=n\cdot P_1$ and for $n\ne 0$ let us define $x_n,y_n\in \Q(z)$ by $P_n=(x_n,y_n)$. Then the following holds:
\begin{itemize}
\item[(i)] For $n\ne 0$ we have the asymptotic formula $\deg(x_n)\sim n^2/2$;
\item[(ii)] For every field $k$ of characteristic $0$, the set of $k(z)$-rational points of $E$ is equal to 
$$
\{P_n : n\in \Z\}.
$$
\end{itemize}
\end{lemma}
\begin{proof} The proof relies on the theory of elliptic surfaces; we refer the reader to \cite{SS} for an introduction. Furthermore, it suffices to work over $\C$.

First, if $P=(a,b)$ is an affine $\C(z)$-rational point in $\Ecal$, then the naive height of $P$ is defined by $h(P)=\deg(a)$ and the canonical height of $P$ is defined by
$$
\hat{h}(P)=\lim_{n\to \infty} \frac{h(2^nP)}{4^n}\in \R_{\ge 0}.
$$
The canonical height is quadratic; in particular, for every integer $n$ we have $\hat{h}(nP)=n^2\hat{h}(P)$. Furthermore, the difference $|h(P)-\hat{h}(P)|$ is bounded by a constant independent of $P$.

Let us now consider the elliptic surface $\pi: X\to \Pro^1$ defined by the Weierstrass equation \eqref{EqnEC}.

The Weierstrass equation  \eqref{EqnEC} is minimal and the elliptic surface is rational (cf. Section 8 in \cite{SS}). One computes that it has three bad fibres of type $I_1$ and one of type $III^*$ (at $z=\infty$). Hence, the structure of its group of sections over $\C$ is given by the entry 43 of the table in the Main Theorem of \cite{OguisoShioda}, which is the lattice $A_1^\vee$. That the section induced by $P_1$ generates is thus reduced to showing that its height is $1/2$, which is a computation left to the reader ---one can use Silverman's algorithm \cite{Silverman} which is explained over $\C(z)$ in \cite{Kuwata}, but note that the canonical height in \emph{loc.cit} is normalized to be half of the one we use here. This proves (ii). 

Item (i) now follows from the quadratic nature of the canonical height and the fact that its difference with the naive height is bounded.
\end{proof}

From the previous lemma we immediately get:

\begin{corollary}\label{CoroQuad} Let $k$ be a field of characteristic $0$. There is a set $D\subseteq k(z)$ which is Diophantine over $k(z)$ and satisfies that the set $\{\deg(f) : f\in D\}\subseteq \N$ is formed by a sequence with quadratic growth. 
\end{corollary}

\begin{proof}[Proof of Theorem \ref{ThmDeg}] Suppose that $\Deg$ is Diophantine over $L(z)$. Let $D$ be as in Corollary \ref{CoroQuad} with $k=L$. Then the set
$$
D'=\{f\in L(z) : \exists g\in D, \deg(f) = \deg(g)\}
$$
is Diophantine over $L(z)$. Since $\{\deg(g) : g\in D\}$ does not contain infinite arithmetic progressions, Corollary \ref{CoroAP} implies that $\ddim(D')$ is finite, but this is false because $\{\deg(g) : g\in D\}$ is an infinite subset of $\N$ and $D'$ is the union of $\{f\in L(z) : \deg f=n\}$ for $n\in \{\deg(g) : g\in D\}$.
\end{proof}

\begin{proof}[Proof of Theorem \ref{ThmN1}] Note that for each $n\ge 1$ we have $L[z]\subseteq \Pcal_n$, so, $\ddim(\Pcal_n)=\infty$. Furthermore, note that for every $\ell$
$$
\Pole_\ell(\Pcal_n)\subseteq Y_{n,\ell} :=\{E\in \Div(\Pro^1) : E\ge 0, \deg(E)=\ell, \mbox{ and }\#\supp(E)\le n\}.
$$
The set $Y_{n,\ell}\subseteq S^\ell\Pro^1$ is closed in the Zariski topology since it is defined by various equations that determine that some of the points in the support of a divisor collide.

Theorem \ref{ThmKollar} gives that there is $P_a\in S^a\Pro^1(L^{\alg})$ and $r\ge 1$ such that 
$$
P_a+r\cdot S^{n+1}\Pro^1\subseteq Y_{n,a+r(n+1)}
$$
But this is not possible: there are divisors in the set of the left hand side with at least $n+1$ points in their support.
\end{proof}

\begin{proof}[Proof of Theorem \ref{ThmVal}] Fix $0< \epsilon \le 1$. The multiplicity of a divisor $E$ on $\Pro^1$ at a point $x$ is denoted by $\mult_x(E)$. 

If $f\in L(z)$ is written in the form $f=p/q$ with $p,q\in L[z]$ coprime, we note that the condition $\deg q\le (1-\epsilon) \deg p$ implies 
$$
\begin{aligned}
\mult_\infty(\Pole(f)) &= \deg(p)-\deg(q)\\
& \ge \epsilon\deg(p)= \epsilon\deg(f)\\
&= \epsilon\deg(\Pole(f)).
\end{aligned}
$$

From this, we observe that 
$$
\Pole_\ell (\Vcal_{\epsilon}(L))\subseteq Z_{\epsilon,\ell}
$$
where
$$
Z_{\epsilon, \ell} = \{E\in \Div(\Pro^1) : E\ge 0, \deg E=\ell, \mbox{ and } \mult_\infty(E)\ge \epsilon\ell\}\subseteq S^\ell\Pro^1.
$$
We note that $Z_{\epsilon,\ell}$ is a Zariski closed subset of $S^\ell\Pro^1$. 

Since $L[z]-\{0\}\subseteq \Vcal_{\epsilon}(L)$ we see that $\ddim(\Vcal_{\epsilon}(L))=\infty$. Thus, Theorem \ref{ThmKollar} gives some $P_a\in S^a\Pro^1(L^{\alg})$ and $r\ge 1$ such that for every $m\ge 1$ we have
$$
P_a+r\cdot S^m\Pro^1\subseteq \overline{\Pole_{a+rm} (\Vcal_{\epsilon}(L))}\subseteq Z_{\epsilon,a+rm}.
$$
Choose a point $Q\ne \infty$ in $\Pro^1$ and let $E_m=P_a+rmQ\in P_a+r\cdot S^m\Pro^1$. Note that $\mult_\infty(E_m)=\mult_\infty(P_a)\le a$. Now choose 
$$
m>\frac{a}{r}\cdot \frac{1-\epsilon}{\epsilon}
$$
and note that since $E_m\in Z_{\epsilon,a+rm}$ we have
$$
\mult_\infty(E_m)\ge \epsilon\deg(E_m)= \epsilon(a+rm)>\epsilon\left(a+a\cdot \frac{1-\epsilon}{\epsilon}\right) = a.
$$
This is a contradiction.
\end{proof}


\section{Koll\'ar's conjecture}\label{SecKollarConjProofs}

\begin{proof}[Proof of Proposition \ref{PropConj}] Item (i) is clear.

For item (ii) notice that $\C[z]\subseteq \Der$ while $1/(z-\lambda)\notin \Der$ for each $\lambda\in \C$.

Item (iii) is clear for $x\ne \infty$, and for $x=\infty$ we take reciprocals.

For item (iv), suppose that $\Ncal$ is Diophantine and observe that the set
$$
\Ncal'=\{f\in \C(z): f\in \C\mbox{ or both }f\mbox{ and }f+4\mbox{ are not squares}\}
$$
is also Diophantine. We get $\C[z]\subseteq \Ncal'$ while for every $\lambda\in \C^\times$ we have $(\lambda z)^{-2}-2+(\lambda z)^2\notin \Ncal'$; this is not possible.

Item (v) is obtained by taking reciprocals.
\end{proof}

\begin{proof}[Proof of Proposition \ref{PropCampana}] 

Suppose that $\ell=\infty$. If $S\ne \emptyset$ we may assume (up to a change of variables in $\C(z)$) that $\infty\in S$. Then $\Ccal_{S,\infty}$ contains $\C[z]$ but its complement is infinite; take for instance the functions $1/(z-\lambda)$ for $\lambda\in \C-S$.

Suppose now that $1<\ell < \infty$. Then $\Ccal_{S,\ell}$ contains an open subset of $\C[z]_{\alpha}$ for every $\alpha \ge \ell$ (namely, the set of polynomials of exact degree $\alpha$). Conjecture \ref{ConjKollar} then gives that the complement of $\Ccal_{S,\ell}$ is finite, which is false since $\ell\ge 2$; again, take the functions $1/(z-\lambda)$ as $\lambda\in \C$ varies.
\end{proof}

\begin{proof}[Proof of Proposition \ref{PropPosChar}] The set $D=k(z)-k(z^p)$ admits the following description: 
$$
D=\left\{ f\in k(z) : \exists f_0,...,f_{p-1}\in k(z)\mbox{ such that }f_i\ne 0 \mbox{ for some }1\le i\le p-1\mbox{ and }f=\sum_{j=0}^{p-1} z^j f_j^p\right\}
$$
from which we see that $D$ is Diophantine in $k(z)$. The set $D$ has the required properties.
\end{proof}

\section{Acknowledgments}

We thank Arno Fehm for answering several questions on large fields, Cec\'ilia Salgado for providing some references on elliptic surfaces, and Thomas Scanlon for asking a question that motivated this research.

N.G.-F. was supported by ANID Fondecyt Regular grant 1211004 from Chile.

H.P. was supported by ANID (ex CONICYT) Fondecyt Regular grant 1190442 from Chile.

This material is based upon work supported by the National Science Foundation under Grant No. DMS-1928930 while the authors participated in the program Definability, Decidability, and Computability in Number Theory, part 2,  hosted by the Mathematical Sciences Research Institute in Berkeley, California, during the Summer of 2022.



\begin{thebibliography}{9}         

\bibitem{Abr} D. Abramovich, \emph{Birational geometry for number theorists}. in Arithmetic geometry, Clay Mathematics Proceedings, vol. 8, American Mathematical Society, 2009, p. 335-373.

\bibitem{AV} D. Abramovich, A. V\'arilly-Alvarado, \emph{Campana points, Vojta's conjecture, and level structures on semistable abelian varieties}. J. Th\'eor. Nombres Bordeaux 30 (2018), no. 2, 525-532.


\bibitem{Anscombe} S. Anscombe, \emph{Existentially generated subfields of large fields}. J. Algebra 517 (2019), 78-94.

\bibitem{BVV} T. D. Browning and K. Van Valckenborgh, \emph{Sums of three squareful numbers}. Exp. Math. 21 (2012), no. 2, 204-211. 

\bibitem{BY} T. D. Browning and S. Yamagishi, \emph{Arithmetic of higher-dimensional orbifolds and a mixed Waring problem}. Math. Z. 299 (2021), no. 1-2, 1071-1101.

\bibitem{Cam} F. Campana, \emph{Fibres multiples sur les surfaces: aspects geom\'etriques, hyperboliques et arithm\'etiques}. Manuscr. Math. 117 (2005), no. 4, p. 429-461.

\bibitem{DPR} M. Davis, H. Putnam, J. Robinson, \emph{The decision problem for exponential diophantine equations}. Ann. of Math. (2) 74 (1961), 425-436.

\bibitem{DemeyerFp} J. Demeyer, \emph{Recursively enumerable sets of polynomials over a finite field are Diophantine}. Invent. Math. 170 (2007), no. 3, 655-670.

\bibitem{DemeyerNF} J. Demeyer, \emph{Diophantine sets of polynomials over number fields}. Proc. Amer. Math. Soc. 138 (2010), no. 8, 2715-2728. 

\bibitem{Denef} J. Denef, \emph{The Diophantine problem for polynomial rings and fields of rational functions}. Trans. Amer. Math. Soc. 242 (1978), 391-399.

\bibitem{DenefZT} J. Denef, \emph{Diophantine sets over $\mathbb{Z}[T]$}. Proc. Amer. Math. Soc. 69 (1978), no. 1, 148-150.


\bibitem{Fehm} A. Fehm, \emph{Subfields of ample fields. Rational maps and definability}. J. Algebra 323 (2010), no. 6, 1738-1744. 

\bibitem{GF} N. Garcia-Fritz, \emph{On the conjectures of Vojta and Campana over function fields with explicit exceptional sets}. Preprint (2022) \url{https://arxiv.org/abs/2203.00626}

\bibitem{KoeZQ} J. Koenigsmann, \emph{Defining $\mathbb{Z}$ in $\mathbb{Q}$}. Ann. of Math. (2) 183 (2016), no. 1, 73-93.

\bibitem{KoeLarge} J. Koenigsmann,  \emph{Defining transcendentals in function fields}. J. Symbolic Logic 67 (2002), no. 3, 947-956.

\bibitem{Kollar} J. Koll\'ar, \emph{Diophantine subsets of function fields of curves}. Algebra Number Theory 2 (2008), no. 3, 299-311. 

\bibitem{Kuwata} M. Kuwata, \emph{The canonical height and elliptic surfaces}. J. Number Theory 36 (1990), no. 2, 201-211.

\bibitem{Mat} J. Matiyasevich, \emph{The Diophantineness of enumerable sets}. (Russian) Dokl. Akad. Nauk SSSR 191 1970 279-282.

\bibitem{MoretBailly} L. Moret-Bailly, \emph{Sur la d\'efinissabilit\'e existentielle de la non-nullit\'e dans les anneaux}.  Algebra Number Theory 1 (2007), no. 3, 331-346. 

\bibitem{OguisoShioda} K. Oguiso, T. Shioda, \emph{The Mordell-Weil lattice of a rational elliptic surface}. Comment. Math. Univ. St. Paul. 40 (1991), no. 1, 83-99. 

\bibitem{Pasten} H. Pasten,  \emph{Notes on the DPRM property for listable structures}. J. Symb. Log. 87 (2022), no. 1, 273-312.


\bibitem{Survey} T. Pheidas, K. Zahidi, \emph{Undecidability of existential theories of rings and fields: a survey.} Hilbert's tenth problem: relations with arithmetic and algebraic geometry (Ghent, 1999), 49-105, Contemp. Math., 270, Amer. Math. Soc., Providence, RI, 2000.


\bibitem{PS} M. Pieropan, D. Schindler, \emph{Hyperbola method on toric varieties}. Preprint (2020) \url{https://arxiv.org/abs/2001.09815}

\bibitem{PSTV} M. Pieropan, A. Smeets, S. Tanimoto, A. V\'arilly-Alvarado,   \emph{Campana points of bounded height on vector group compactifications}. Proc. Lond. Math. Soc. (3) 123 (2021), no. 1, 57-101. 

\bibitem{Poonen} B. Poonen, \emph{The set of nonsquares in a number field is Diophantine}. Math. Res. Lett. 16 (2009), no. 1, 165-170. 

\bibitem{PopLarge} F. Pop, \emph{Little survey on large fields--old \& new.}  Valuation theory in interaction, 432-463, EMS Ser. Congr. Rep., Eur. Math. Soc., Z\"urich, 2014. 


\bibitem{RTW} E. Rousseau, A. Turchet, J. T.-Y. Wang, \emph{Nonspecial varieties and generalised Lang-Vojta conjectures}. Forum Math. Sigma 9 (2021), Paper No. e11, 29 pp.

\bibitem{SS} M. Sc\"utt, T. Shioda, \emph{Elliptic surfaces}. Algebraic geometry in East Asia-Seoul 2008, 51-160, Adv. Stud. Pure Math., 60, Math. Soc. Japan, Tokyo, 2010.

\bibitem{Silverman} J. Silverman, \emph{Computing heights on elliptic curves}. Math. Comp. 51 (1988), no. 183, 339-358. 

\bibitem{VV} K. Van Valckenborgh, \emph{Squareful numbers in hyperplanes}. Algebra Number Theory 6 (2012), no. 5, 1019-1041.

\bibitem{Vojta} P. Vojta, \emph{Diophantine approximation and Nevanlinna theory}. Arithmetic geometry, 111-224, Lecture Notes in Math., 2009, Springer, Berlin, 2011.


\end{thebibliography}
\end{document}